\documentclass[12pt,reqno]{amsart}
\pdfoutput=1
\usepackage[headings]{fullpage}
\usepackage{amsfonts}
\usepackage{amsthm}
\usepackage{amssymb}
\usepackage{epstopdf}
\usepackage{graphicx}
\usepackage{texdraw}
\usepackage{url}
\usepackage[all]{xy}
\usepackage{pb-diagram}
\usepackage{float}   
\usepackage[bookmarks=true,%
    colorlinks=true,%
    linkcolor=blue,%
    citecolor=blue,%
    filecolor=blue,%
    menucolor=blue,%
    urlcolor=blue,%
    breaklinks=true]{hyperref}

\newtheorem{theorem}{Theorem}[section]
\theoremstyle{definition}
\newtheorem{proposition}[theorem]{Proposition}
\newtheorem{lemma}[theorem]{Lemma}

\newtheorem{remark}[theorem]{Remark}

\def\be{\begin{equation}}
\def\ee{\end{equation}}
\def\+{\,+\,}
\def\m{\,-\,}
\def\={\;=\;}
\def\mod{\;\text{\rm mod}\;}

\def\Q{\mathbf Q}   
\def\Z{\mathbf Z}   
\def\N{\mathbf N}   
\def\C{\mathbf C}   
\def\R{\mathbf R}   
\def\P{\mathbf P}

\def\i{^{-1}}

\def\a{\alpha} 
 
\def\ve{\varepsilon}

\def\BN{\N}
\def\BZ{\Z}
\def\BQ{\Q}
\def\BR{\R}
\def\BC{\C}

\def\calH{\mathcal H}

\def\ep{\epsilon}
\def\vepsilon{\varepsilon}
\def\ve{\vepsilon}
\def\ti{\widetilde}

\def\SL{\mathrm{SL}}

\def\z{\zeta}

\def\th{\theta}
\def\diag{\mathrm{diag}}

\def\e{\bold e}  
\def\Li{\mathrm{Li}}
\newcommand{\mb}{\mathbf}

\def\sm#1#2#3#4{\bigl(\begin{smallmatrix}#1&#2\\#3&#4\end{smallmatrix}\bigr)}

\def\Z{\Bbb Z} 
\def\Q{\Bbb Q} 
\def\R{\Bbb R} 
\def\C{\Bbb C}  
\def\N{\Bbb N} 
\def\O{\text O}
\def\={\;=\;} 
\def\ap{\;\sim\;} 
\def\+{\,+\,} 
\def\z{\zeta} 
\def\L{\mathrm{L}} 
\def\Ps{{}_1\Psi_1}

\def\bfI{{\bf I}}
\def\diag{\mathrm{diag}}
\def\PSL{\mathrm{PSL}}
\def\wtA{\widetilde A}

\begin{document}


\title{Asymptotics of Nahm sums at roots of unity}
\author{Stavros Garoufalidis}
\address{School of Mathematics \\
         Georgia Institute of Technology \\
         Atlanta, GA 30332-0160, USA \newline 
         {\tt \url{http://www.math.gatech.edu/~stavros}}
         }
\email{stavros@math.gatech.edu}
\author{Don Zagier}
\address{Max Planck Institute for Mathematics \\
         53111 Bonn, Germany \newline
         {\tt \url{http://people.mpim-bonn.mpg.de/zagier}}
         }
\email{dbz@mpim-bonn.mpg.de}


\begin{abstract}
We give a formula for 
the radial asymptotics to all orders of the special $q$-hypergeometric 
series known as Nahm sums at complex roots of unity. This result is used 
in~\cite{CGZ} to prove one direction of Nahm's conjecture relating the 
modularity of Nahm sums to the vanishing of a certain invariant in $K$-theory.
The power series occurring in our asymptotic formula are identical to the 
conjectured asymptotics of the Kashaev invariant of a knot once we convert 
Neumann-Zagier data into Nahm data, suggesting a deep connection between 
asymptotics of quantum knot invariants and asymptotics of Nahm sums 
that will be discussed further in a subsequent publication.
\end{abstract}

\maketitle

{\footnotesize
\tableofcontents
}




\section{Introduction}

Nahm sums are special $q$-hypergeometric series whose summand involves
a quadratic form, a linear form and a constant. They were introduced
by Nahm~\cite{Nahm} in connection with characters of rational conformal 
field theories. 
Nahm formulated a very surprising conjecture, that has 
elicited a lot of interest, relating the question of their modularity to 
the vanishing of a certain invariant in algebraic $K$-theory (more 
specifically, in $K_3$-group, or equivalently the Bloch group, 
of the algebraic numbers). 
This conjecture, at least in one direction, is proved in~\cite{CGZ} using 
the asymptotics given in this paper together with the construction of units
associated to elements of $K$-theory given there.  

The definition of Nahm sums and the question of determining when they are
modular were motivated by the famous Rogers-Ramanujan identities, which say 
that
$$ 
G(q)\,:=\,\sum_{n=0}^\infty\frac{q^{n^2}}{(q)_n}
 \,=\,\prod_{\substack{ n>0 \\ (\frac n5)=1}}\!\frac1{1-q^n},\qquad
H(q)\,:=\,\sum_{n=0}^\infty\frac{q^{n^2+n}}{(q)_n}
 \,=\,\prod_{\substack{ n>0 \\ (\frac n5)=-1}}\!\frac1{1-q^n}\,,
$$
where $(q)_n=(1-q)\cdots(1-q^n)$ is the $q$-Pochhammer symbol or quantum 
$n$-factorial. These identities imply via the Jacobi triple product formula 
that the two functions $q^{-1/60}G(q)$ and $q^{11/60}H(q)$ are quotients of 
unary theta-series by the Dedekind eta-function and hence are modular 
functions. More generally, Nahm~\cite{Nahm} 
considered the following multi-dimensional generalization
\begin{equation}
\label{eq.FABC}
F_Q(q) \= F_{A,B,C}(q) \= 
\sum_{n\in\Z_{\ge0}^N} \frac{q^{Q(n)}}{(q)_{n_1}\cdots(q)_{n_N}}
\qquad \in \quad \BZ(\!(q^{\frac1{d}})\!)\,,
\end{equation}
where $Q:\Z^N\to\Q$ is a quadratic function, i.e., a function of the form 
$$ 
Q(n) \= \frac12 n^tAn + Bn+C 
$$
where $A=(a_{ij})$ is a symmetric positive definite $N \times N$ matrix with 
rational entries, $B \in \BQ^N$ a column vector and $C \in \BQ$ a scalar and
$d$~is any
denominator of~$Q$ (i.e., any positive integer with $dQ(\Z^N)\subseteq\Z$).

Our aim is to give the asymptotic expansion of $F_Q(q)$ as $q$ approaches 
a root of unity (of order prime to a denominator of~$Q$) radially.



The constant term of these asymptotic expansions is 
used in~\cite{CGZ} to prove one direction of Nahm's modularity conjecture. 
Very strikingly, our formulas are identical to a collection of power series 
(one for every complex root of unity) associated to a Neumann-Zagier datum
in~\cite{DG1,DG2} and conjectured to be the asymptotic expansion of the
Kashaev invariant at complex roots of unity. The coincidence of the
asymptotics of Nahm sums at $q=1$ and the series of~\cite{DG1} was observed
several years ago via an explicit map from Neumann-Zagier data to Nahm data, 
and leads to a deeper connection between quantum invariants of knots defined 
on the roots of unity (such as the Kashaev invariant) and $q$-series 
invariants of knots (such as the 3D-index of 
Dimofte-Gukov-Gaiotto~\cite{DGG1}). This connection will be explained in
a later publication~\cite{GZ:qseries}.  

Nahm sums appear naturally in cohomological
mirror symmetry~\cite{KoSo}, in the representation theory of 
quivers~\cite{knots-quivers} and in quantum 
topology in relation to the stabilization of the coefficients of the 
colored Jones polynomial~\cite{GL:Nahm}. In addition
they are building blocks of the 3D-index of an ideally triangulated 
manifold due to Dimofte-Gaiotto-Gukov~\cite{DGG1,DGG2}, and appear as
holomorphic blocks in the state-integrals of Chern-Simons theory
with complex gauge group~\cite{holomorphic-blocks,GK:qseries}. 
Further connections between quantum topological invariants and Nahm 
sums are given in~\cite{GZ:qseries}.


\section{Asymptotic formula for the summand of a Nahm sum}
\label{sec.results}

Recall the Pochhammer symbol $(qx;q)_\infty\=\prod_{i \geq 1} (1-q^i x)$,
an entire function of $x$ for $q$ a complex number with $|q|<1$. 
Lemma~\ref{lem.1} below gives the radial asymptotics of the 
Pochhammer symbol at roots of unity. To formulate it, recall 
the $r$th Bernoulli polynomial $B_r(x)$, the $r$th polylogarithm function
$\Li_r(w)=\sum_{k \geq 1} w^k/k^r$ for $|w|<1$ and the 
{\it cyclic quantum dilogarithm} function 
\begin{equation}
\label{eq.Dmz}
D_{\zeta}(x) \= \prod_{t=1}^{m-1}(1-\zeta^t x)^t \quad\in\Q(\z)[x]
\end{equation}
where $\zeta$ is a primitive $m$th root of unity. 
The function that we will actually use is $D_\z(x)^{1/m}$ when $|x|<1$, 
where the $m$th root is defined by using the principal part of the 
logarithm of each factor in~\eqref{eq.Dmz}.
Below, if $f(\ve)$ is the germ of a smooth function of $\ve$ defined in
a neighborhood of~0 in the right half-plane $\Re(\ve)>0$, we write that 
\begin{equation} 
\label{eq.fas}
f(\ve) \ap \, \sum_{k=0}^\infty a_k\,\ve^k  
\end{equation} 
for $\ve\searrow0$ if $f(\ve)=\sum_{k=0}^{K-1}a_k\ve^k+\text O(\ve^K)$ for 
every $K>0$ as $\ve$ tends to~0 from the right, or equivalently if $f$ is 
$C^\infty$ from the right at~0 with Taylor coefficients $a_k=f^{(k)}(0)/k!\,$. 


\begin{lemma}
\label{lem.1}
Let $w$ be a complex number with $|w|<1$, $q=\z e^{-\ve/m}$ where $\z$ is
a primitive $m$th root of unity, and $\nu$ a complex number such that
$\nu \ve =o(1)$. Set
\begin{align}
\label{eq.l1}
\log(q \,w \, e^{-\frac{\nu \ve}{m}};q)_\infty & \= -\frac{1}{m\ve} \Li_2(z) - 
\Bigl(\frac{\nu}{m}-\frac{1}{2}\Bigr) \log (1-z) -\frac{\ve \nu^2}{2m} 
\frac{z}{1-z}  \\ \notag
& \qquad - \frac{1}{m} \log D_\z(w) - \log(1-w) + \psi_{w,\z}(\nu,\ve)
\end{align}
where $z=w^m$. Then, $\psi_{w,\z}(\nu,\ve)$  has an explicit asymptotic 
expansion in $\BC[\nu \ve^{2/3}][[\ve^{1/3}]] \subset \BC[\nu][[\ve]]$ 
as $\ve \searrow 0$, 
\begin{equation}
\label{eq.psi}
\psi_{w,\z}(\nu,\ve) \sim -\sum_{r \geq 2} \sum_{t=1}^m 
\Bigl( B_r\Bigl(1-\frac{t+\nu}{m} \Bigr) -\delta_{r,2} \frac{\nu^2}{m^2}
\Bigr) \Li_{2-r}(\z^t w) \frac{\ve^{r-1}}{r!}
\end{equation}
in which the coefficient of $\nu^n$ is $O(\ve^{2n/3})$
for every $n \geq 0$. 
\end{lemma}

Fix a positive definite $N \times N$ matrix $A$ with rational entries
and let $(z_1,\dots,z_N)$ denote the unique solution in~$(0,1)^N$ 
of {\em Nahm's equation}
\begin{equation}
\label{eq.nahm}
1 \m z_i \=  \prod_{j=1}^N z_j^{A_{ij}} \qquad (i=1,\dots,N) \,.
\end{equation}

\noindent
We define a real number
\be 
\label{alph} 
\Lambda \= -\sum_{j=1}^N \L(z)\,, 
\ee
where $\L(z)$ is the {\em Rogers dilogarithm function} (shifted by a 
constant to make $\L(1)=0$), defined for $0<z<1$ by 
\begin{equation}
\label{eq.rogers}
\L(z) \= \Li_2(z) \+\frac12\,\log(z)\,\log(1-z) \,-\, \frac{\pi^2}{6}\,.
\end{equation}
Let
\begin{equation}
\label{eq.Atilde}
\wtA \= A\+\diag(z/(1-z))
\end{equation}
where $\diag(z/(1-z))$ denotes the diagonal matrix with diagonal elements 
$z_i/(1-z_i)$. 

\begin{proposition}
\label{prop.1}
Fix $k=(k_1,\dots,k_N) \in (\BZ/m\BZ)^N$.
Consider natural numbers $n_i \in \BN$ for $i=1,\dots,N$ satisfying 
$n_i \equiv k_i \bmod m$ for $i=1,\dots,N$ and write
$n_i=\frac{1}{\ve}\log \frac{1}{z_i}
+ \frac{1}{\sqrt{\ve}}x_i$ where $z=(z_1,\dots,z_N) \in (0,1)^N$ is 
the distinguished solution of Nahm's equation and $q=\z e^{-\ve/m}$.
Then, 
\begin{align}
\label{eq.p1}
\frac{e^{-\ve Q(n)/m}}{\prod_{i=1}^N (q)_{n_i}} & \= 
\Bigl(\frac{\ve}{2\pi}\Bigr)^{\frac{N}{2}}\,
e^{\frac{\Lambda}{m\ve}}\,\prod_{i=1}^N \th_i^{B_i}(1-z_i)^{\frac{1}{2}-\frac{1}{m}}\, 
\prod_{i=1}^N D_{\z}(\th_i)^{-\frac1m}\, 
\prod_{i=1}^N\frac{\th_i^{(A k)_i}}{(\th_i;\z)_{k_i}}
\\
& \qquad e^{-\frac{1}{m} x^t \ti A x}  e^{\frac{1}{m} B^t x \ve^{1/2} -\frac{1}{m} C \ve}
\prod_{i=1}^N \psi_{\z^{k_i} z_i^{\frac{1}{m}},\z}(x_i \, \ve^{-1/2},\ve)
\end{align}
where $\th_i=z_i^{1/m_i} \in (0,1)$. 
When $|x_i| \leq \ve^{\lambda-\frac{1}{6}}$ for some $\lambda >0$, 
then each term of the product of the $\psi$-terms in~\eqref{eq.p1} has an 
asymptotic expansion in $\BC[x][[\ve^{1/2}]]$ in which the coefficient 
of $x^n$ is $O(\ve^n)$. 
\end{proposition}


\section{Asymptotic formula for a Nahm sum}
\label{sec.asn}

The Nahm sum
$F_Q(q)$ is a formal Puiseux series with integer coefficients in the variable
$q^\frac{1}{d}$, (where $d$ is a denominator of $Q$)
analytic in a finite covering of the punctured open 
unit disk $0<|q|<1$. It will be convenient to work with the
complex-valued function $f_Q=f_{A,B,C}$ defined in the upper half-plane by
\begin{equation}
\label{eq.fQ}
f_{Q}(\tau) \= f_{A,B,C}(\tau) \= F_Q(e^{2 \pi i\tau}), \quad \Im(\tau)>0\,,
\end{equation}
with the convention that $(e^{2 \pi i\tau})^\lambda =e^{2 \pi i\tau \lambda}$ 
for any~$\lambda\in\Q$. Our main Theorem~\ref{thm.1} concerns 
the asymptotic expansion of $f_{Q}(\tau)$ at the cusps, i.e., when 
$\tau$ approaches a rational number from above. To formulate our results,
we need to introduce some more ingredients, namely the quadratic Gauss
sums (which appear in the constant term of the asymptotics),
and the formal Gaussian integration which gives the asymptotics 
to all orders.

Recall the formal Gaussian integration of an analytic 
function $f(x)$ in $N$ variables $x=(x_1,\dots,x_N) \in \BC^N$ with values
in a power series ring, following the notation of~\cite{Zagier} 
\begin{equation}
\label{eq.FGI}
\bfI_{A} \bigl[f \bigr] \=
\frac{\int_{\BR^N} e^{-\frac{1}{2} x^t A x} f(x) dx}{
\int_\BR e^{-\frac{1}{2} x ^t A x} dx} 
\;\,\sim\;\, \sum_{n \geq 0} \frac{1}{2^n n!} (\Delta^n_A f)(0)
\,,
\end{equation}
where $\Delta_A$ denotes the Laplacian with respect to the quadratic form
$x^t A x$. In particular, for $1\times1$ matrices~$A$, the formal 
Gaussian integration is given explicitly by: 
\begin{equation}
\label{eq.FGI2}
\bfI_{A} \Biggl[\,\sum_{j=0}^\infty c_j x^j \Biggr] 
\quad\sim\quad \sum_{\ell=0}^\infty (2\ell-1)!! \, c_{2\ell} \, A^{-\ell} \,.  
\end{equation}
This expression is meaningful if the $c_n$ belong to some power series ring
(such as $\BC[[\ve^{1/2}]]$) and the valuation of $c_n$ approaches zero as 
$n$ tends to infinity.

For our asymptotic formulas, we use the formal Gaussian integration
(motivated by Proposition~\ref{prop.1})
\begin{equation}
\label{eq.Ik}
I_{Q,\z}(k,\ve) \= \bfI_{\frac{1}{m}\wtA} \biggl[
e^{\frac{1}{m} x^t B \ve^{1/2} -\frac{1}{m} C \ve}
\prod_{i=1}^N \psi_{\z^{k_i} z_i^{\frac{1}{m}},\z}(x_i \, \ve^{-1/2},\ve) \biggr]
\qquad (k \in (\BZ/m\BZ)^N)
\end{equation} 
which is a formal power series in $\ve$ with complex coefficients, 
where we think of the function inside the argument of $\bfI$ as an 
analytic function of $x=(x_1,\dots,x_N) \in \BC^N$ with values in the
power series ring $\BC[[\ve^{1/2}]]$. 

We call a positive integer $D$ a strong denominator of $Q$ if the value 
of $Q(k)$ modulo~1 for $k\in\Z^N$ depends only on the residue class of 
$k$ modulo~$D$.  (For instance, one can take $D=2d$ where $d$ is any 
common denominator of $Q$, i.e., any integer with $dQ(\Z^N)\subseteq\Z$.) 
The final ingredient is the quadratic Gauss sum 
\begin{equation}
\label{eq.Gab}
G(Q,\a) 
 \= \frac{1}{D^N} \sum_{k \in (\BZ/D \BZ)^N} 
\e\bigl( \overline{\a} \, Q(k)\bigr) \,, 
\end{equation}
where 
\be
\label{eq.ex}
\e(x)=e^{2 \pi i x} \,,
\ee
$\a \in\BQ$ is prime to $D$ 
and $\overline{\a}$ denotes the reduction of $\a$ modulo $D$ 
and $D$ is any strong denominator of $Q$.
The sum on the right is
clearly independent of the choice of~$D$.


We now have all the ingredients to formulate our main theorem concerning the
radial asymptotics of Nahm sums at roots of unity. Since 
$f_{A,B,C}(\tau)=\e(C \tau) f_{A,B,0}(\tau)$, we can restrict to the case 
$C=0$.



\begin{theorem}
\label{thm.1} 
Let $Q(x)=\frac12x^tAx+Bx$ be a quadratic function from $\Z^N$ to~$\Q$ 
as above. Fix a rational number $\a$ whose denominator $m$ is odd and
prime to some denominator of $Q$, and let $\z=\e(\a)$ denote the corresponding
primitive $m$-th root of unity.
Let $\th_i=z_i^{1/m} \in (0,1)$, where $(z_1,\dots,z_n)$
as in~\eqref{eq.nahm} is the positive real solution of the Nahm equation 
for~$A$. Then the asymptotics of the function $f_Q(\tau)$ defined 
in~\eqref{eq.fQ} as $\tau$ tends to $\a\in\Q$ is given by
\be 
\label{eq.FGIII}
e^{-\frac{\Lambda}{m \ve}}\, f_{Q}\Bigl(
\a +\frac{i\ve}{2 \pi m}\Bigr)
\;\, \sim \;\,  
\frac{\chi^N}{m^{N/2}} \,c(Q)\,G(Q,\a)\,S_{Q,\z}(\ve)
\quad\;\text{\rm as $\ve\searrow0\,$,} 
\ee
where $\Lambda\in\R$ is defined by~\eqref{alph}, 
$\chi$ is the 12th root of~$\z$ defined by 
$\chi= \e\bigl(\binom{m-1}{2}\frac{\a}{12}\bigr)$, 
$c(Q)$ is defined by
\be
\label{eq.CAB}
c(Q) \= \det(\wtA)^{-\frac12} \,
\prod_{i=1}^N \th_i^{B_i}(1-z_i)^{\frac{1}{2}-\frac{1}{m}}
\ee
with $\wtA$ (a positive definite matrix with positive determinant)
as in~\eqref{eq.Atilde}, $G(Q,\a)$ is the Gauss 
sum~\eqref{eq.Gab} and $S_{Q,\z}(\ve)$ 
is the formal power series in~$\ve$ defined by
\be 
\label{eq.SQz}
S_{Q,\z}(\ve) \= \prod_{i=1}^N D_{\z}(\z\th_i)^{-\frac1m}\,
\sum_{k\in(\Z/m\Z)^N}
\z^{\overline{Q(k)}}\,
\prod_{i=1}^N\frac{\th_i^{(A k)_i}}{(\z\th_i;\z)_{k_i}}
\,I_{Q,\z}(k,\ve)
\ee
with $I_{Q,\z}(k,\ve)$ as in~\eqref{eq.Ik}, 
where $\overline{Q(k)}$ denotes the reduction of $Q(k)$ modulo $m$.
Moreover, we have 
\begin{equation} 
\label{eq.PS}
S_{Q,\z}(\ve)^m \;\in\;  F_m[[\ve]]
\end{equation}
where $F=\BQ(z_1^{\frac{1}{d}},\dots,z_N^{\frac{1}{d}})$ 
and $F_m$ is the cyclotomic extension $F(\zeta)$ of~$F$. 
\end{theorem}

It is the final statement~\eqref{eq.PS} of this theorem, restricted to $\ve=0$, 
that is used in~\cite{CGZ} to prove one direction of Nahm's Modularity 
Conjecture.

\begin{remark}
\label{rem.q=1}
When $\zeta=1$, the statement and the proof of Theorem~\ref{thm.1} is valid
when $A$, $B$ and $C$ have real (but not necessarily, rational) entries.
\end{remark}





\section{Proof of the asymptotic formulas}
\label{sec.radial.nahm}

\subsection{Proof of Lemma~\ref{lem.1}}

In this section we give the proof of Lemma~\ref{lem.1}.

\begin{proof}
We have
\begin{align*}
- \log(q \,w \, e^{-\frac{\nu \ve}{m}};q)_\infty &=
-\sum_{n \geq 1} \log(1-q^n \,w \, e^{-\frac{\nu \ve}{m}}) 
& \\
& =
\sum_{k \geq 1} \frac{1}{k} \sum_{n \geq 1} (q^n \,w \, e^{-\frac{\nu \ve}{m}})^k 
& \\
&\sim \sum_{k \geq 1} \sum_{t=1}^m \frac{(\z^t w)^k}{k} 
\frac{e^{-k(\nu+t)\ve/m}}{1-e^{-k \ve}} & \text{sum by $n \equiv t \bmod m$}
\\
&= \sum_{k \geq 1} \sum_{t=1}^m \frac{(\z^t w)^k}{k} \sum_{r \geq 0}
B_r \Bigl(1-\frac{t+\nu}{m} \Bigr) \frac{(k \ve)^{r-1}}{r!} 
& \text{definition of $B_r(x)$} 
\\
&= \sum_{r \geq 0} \sum_{t=1}^m B_r \Bigl(1-\frac{t+\nu}{m} \Bigr) 
\Li_{2-r}(\z^t w) \frac{\ve^{r-1}}{r!} & \text{definition of $\Li_{2-r}(z)$} \,.
\end{align*}
Using the distribution property
$$
\sum_{t=1}^m \Li_r (\z^t w)=m^{1-r} \,\Li_r(w^m)
$$
for the polylogarithm, we see that the $r=0$ and $r=1$ terms are given by
$$
\frac{1}{\ve} \sum_{t=1}^m \Li_2(\z^t w) = \frac{1}{m\ve} \Li_2(w^m) 
$$
and
$$
\sum_{t=1}^m \Bigl(\frac{1}{2}-\frac{\nu+t}{m}\Bigr) \Li_1(\z^t w) =
\Bigl(\frac{\nu}{m}-\frac{1}{2}\Bigr) \log(1-w^m) + \frac{1}{m} D_\z(w)
+\log(1-w)
$$
respectively, and that the $r=2$ term is given by
$$
\frac{\ve}{2} \frac{\nu^2}{m^2} \sum_{t=1}^m \Li_0(\z^t w) =
\frac{\ve \nu^2}{2m} \frac{w^m}{1-w^m} \,.
$$
The result follows.
\end{proof}

\subsection{Proof of Proposition~\ref{prop.1}}

In this section we give the proof of Proposition~\ref{prop.1}. 
Let $n_i$ be as in the statement of Proposition~\ref{prop.1}
and $q=\z e^{-\ve/m}$. It follows that $q^{n_i}= w_i e^{-\nu_i \ve/m}$
where $w_i=\zeta^{k_i} \th_i$, $\th_i=z_i^{\frac{1}{m}}$ 
and $\nu_i=x_i \ve^{-1/2}$.
Therefore,
\begin{equation}
\label{eq.qni}
\frac{1}{(q)_{n_i}}= \frac{(q^{n_i+1};q)_\infty}{(q;q)_\infty} =
\frac{(q w_i  e^{-\nu_i \ve/m};q)_\infty}{(q;q)_\infty}
\end{equation}
For the expansion of the denominator,  
the modularity of $\eta(z)$ (or alternatively, the Euler-macLaurin formula) 
implies that when $q=e^{-\ve}$ with $\ve \searrow 0$, we have:
\begin{equation}
\label{eq.etae}
\log \Bigl(\frac{1}{(q;q)_\infty} \Bigr) = \frac{\pi^2}{6\ve} 
-\frac{1}{2} \log \Bigl(\frac{2 \pi}{\ve} \Bigr) -\frac{\ve}{24} + O(\ve^K)
\end{equation}
for all $K >0$. For the expansion of the numerator in~\eqref{eq.qni}, 
we use Lemma~\ref{lem.1} combined with the following identity:
\begin{equation}
\label{eq.Dzz}
D_\z(\z^{k_i} \th_i)= \frac{(\th_i;\z)_{k_i}^m D_\z(\th_i)}{(1-z_i)^{k_i}} =
\frac{(\th_i;\z)_{k_i}^m D_\z(\th_i)}{\prod_{j=1}^N \th_j^{m (A k)_i}} \,,
\end{equation}
where the first equality follows form the fact that
$$
\frac{D_{\zeta}(\zeta x)}{D_{\zeta}(x)}=\frac{(1-x)^m}{1-x^m} \,
$$ 
and the second equality follows from the fact that $z$ is a solution to
Nahm's equation.

Finally, the quadratic form expands as follows:
\begin{align*}
Q(n) &= 
Q\bigl(\frac{1}{\ve}\log \frac{1}{z} + \frac{x}{\sqrt{\ve}}\bigr)
\\
&= \frac{1}{2\ve^2} (\log z)^t A \log z - \frac{1}{\ve \sqrt{\ve}}
x^t A \log z + \frac{1}{2 \ve} x^t A x \\
& \quad\, - \frac{1}{\ve} B^t \log z +\frac{1}{\sqrt{\ve}} B^t x + C \,. 
\end{align*}
Using the fact that $z$ satisfies Nahm's equation~\eqref{eq.nahm}, 
it follows that
\begin{align*}
-\frac{\ve}{m} Q(n) &= -\frac{1}{2 m \ve} \log z \cdot \log(1-z) +
\frac{1}{m} B^t \log z -\frac{\ve}{m} C \\
& \quad\, - \frac{1}{2m} x^t A x 
+ \frac{1}{m\sqrt{\ve}} x^t \log(1-z) + \frac{1}{m} \sqrt{\ve} B^t x
\,.
\end{align*}
The first term in the first line of the above equation converts the
dilogarithm by the Rogers dilogarithm. 
The middle term of the last line of the above equation cancels with
one term of~\eqref{eq.l1}. The remaining terms combine to 
conclude~\eqref{eq.p1}. This concludes the proof of Equation~\eqref{eq.p1}
in Proposition~\ref{prop.1}.

Fix $\lambda >0 $ and let $|x| \leq
\ve^{\lambda-\frac{1}{6}}$. Then, we can use the asymptotic 
expansion~\eqref{eq.psi} of $\psi$ and conclude the claim of the
proposition. 
\qed

\subsection{Proof of Theorem~\ref{thm.1}}
\label{sub.thm1}

In this section we give a proof of Theorem~\ref{thm.1}. Our strategy is
to split Nahm sums according to congruence classes in which case their 
summand is a positive real number with a unique peak, and their asymptotics 
can be studied using several applications of the Poisson summation formula. 

Below, $\zeta$ denotes a primitive $m$th root of unity coprime to
$D$, a strong denominator of $Q$. Denote by $a_n(q)$ the summand 
of~\eqref{eq.FABC} for $n=(n_1,\dots,n_N) \in \BZ_{\ge0}^N$.Clearly, we can
split $F_Q(q)$ as
\begin{equation}
\label{eq.Fcong}
F_{Q}(q) = \sum_{k \in (\BZ/m\BZ)^N, \,\, k' \in (\BZ/d\BZ)^N} 
F_{Q}^{[k,k']}(q) \,,
\end{equation}
where 
\begin{equation}
\label{eq.Fkk}
F_{Q}^{[k,k']}(q)=\sum_{\substack{ n\in\Z_{\ge0}^N \\ 
n \equiv k \bmod m, \,\, n \equiv k' \bmod D}} a_{n}(q) 
\qquad \text{and} \qquad a_n(q)= \frac{q^{Q(n)}}{(q)_{n_1}\cdots(q)_{n_N}}
\,.
\end{equation}
When $n$ is in a fixed congruence class modulo $mD$, with $n=k \bmod m$
and $n=k' \bmod d$, then $Q(n)$ takes a fixed value modulo $1/D$ and 
using the Chinese remainder theorem, we get:
$$
\z^{Q(n)}=\e(\overline{\overline{\a}} Q(k')) \,\z^{\overline{Q(k)}} 
$$
where $\overline{x}$ and $\overline{\overline{x}}$ denote the reduction of $x$
modulo $m$ and $D$ respectively.

Set $f_{Q}^{[k,k']}(\tau)=F_{Q}^{[k,k']}(\e(\tau))$. When 
$\tau=\a+\frac{i \ve}{2 \pi m}$ (i.e., $q=\z e^{-\ve/m}$), it follows that
\begin{equation}
\label{eq.fkk}
f_{Q}^{[k,k']}(\tau) = \e(\overline{\overline{\a}} Q(k')) \,\z^{\overline{Q(k)}}
f_{Q,\z}^{[k,k']}(\ve)
\end{equation}
where
\begin{equation}
\label{eq.fzkk}
f_{Q,\z}^{[k,k']}(\ve) =
\sum_{\substack{ n\in\Z_{\ge0}^N \\ 
n \equiv k \bmod m, \,\, n \equiv k' \bmod D}} 
a^+_n(\z e^{-\ve/m}) e^{-\ve Q(n)/m}, \qquad
a^+_n(q)= \frac{1}{(q)_{n_1}\cdots(q)_{n_N}} \,.
\end{equation} 
Recall the definition of $f(\ve) \ap g(\ve)$ from~\eqref{eq.fas}.

{\bf Claim 1:} We have: 
\begin{equation}
\label{eq.fzkksim}
f_{Q,\z}^{[k,k']}(\ve) \ap f_{Q,\z}^{[k,0]}(\ve) \,.
\end{equation}
This follows from an application of the Poisson summation formula
discussed below. Assuming this, it follows that

\begin{equation}
\label{eq.fzkkD}
\sum_{k' \bmod D} f_{Q,\z}^{[k,k']}(\ve) =\frac{1}{D} f_{Q,\z}^{[k]}(\ve)
\end{equation}
where
$$
f_{Q,\z}^{[k]}(\ve)\=
\sum_{\substack{ n\in\Z_{\ge0}^N \\ 
n \equiv k \bmod m}} a^+_n(\z e^{-\ve/m}) e^{-\ve Q(n)/m} \,.
$$
We now write
$$
f_Q(\tau) = F_Q(e^{2\pi i\tau}) = \sum_{n\in\Z_{\ge0}^N} a_n(e^{2\pi i\tau}) 
e^{2 \pi i Q(n)\tau} \,.
$$
Notice that $a_n(\e(\tau))$ depends on $\tau$ modulo 1. Now using a strong
denominator $D$ of $Q$ and 
Combining~\eqref{eq.Fcong}, \eqref{eq.fkk}, \eqref{eq.fzkksim} and
\eqref{eq.fzkkD}, we split 

\begin{equation}
\label{eq.fQzcomb}
f_{Q}\Bigl(\a +\frac{i\ve}{2 \pi m}\Bigr)
\;\, \sim \;\,  G(Q,\a) \sum_{k \in (\BZ/m\BZ)^N} \z^{\overline{Q(k)}}
f_{Q,\z}^{[k]}(\ve) \,.
\end{equation}

We now study in detail the asymptotics of $f_{Q,\z}^{[k]}(\ve)$
as $\ve \searrow 0$ for fixed $k \in (\BZ/m\BZ)^N$.
The asymptotic analysis uses Proposition~\ref{prop.1} (which describes
a unimodal property of the summand of $f_{Q,\z}^{[k]}(\ve)$) and the Poisson 
summation formula applied several times described terse in p.53--54
of~\cite{Zagier} and in much more detail in~\cite{VZ}.

Let $\frac{1}{\sqrt{\ve}} x_i = n_i -\frac{1}{\ve} \log \frac{1}{z_i}$
and set $x^{(0)}_i(\ve)=\left\lfloor -\frac{1}{\ve} \log \frac{1}{z_i} 
\right\rfloor + k_i$. So, if $n_i = k_i \bmod m$, then 
$x_i \in (x^{(0)}_i(\ve)+m \BZ)\sqrt{\ve}$. Using Proposition~\ref{prop.1},
and extending $a^+_n(q)=0$ for $n \in \BZ^N\setminus \BZ^N_{\geq0}$,
it follows that
$$
f_{Q,\z}^{[k]}(\ve) \= \gamma_{Q,\z}(\ve)\, 
\sum_{x \in (x^{(0)}(\ve)+m \BZ)\sqrt{\ve}} e^{-\frac{1}{m} x^t \ti A x}
\varphi(x,\ve)
$$
where 
$$
\gamma_{Q,\z}(\ve) \=
\Bigl(\frac{\ve}{2\pi}\Bigr)^{\frac{N}{2}}\,
e^{\frac{\Lambda}{m\ve}}\, 
\prod_{i=1}^N \th_i^{B_i}(1-z_i)^{\frac{1}{2}-\frac{1}{m}}\, 
\prod_{i=1}^N D_{\z}(\th_i)^{-\frac1m}\, 
\prod_{i=1}^N\frac{\th_i^{(A k)_i}}{(\th_i;\z)_{k_i}} e^{-C\ve/m} 
$$
and
$$
\varphi(x,\ve) =
e^{\frac{1}{m} B^t x \ve^{1/2}}
\prod_{i=1}^N \psi_{\z^{k_i} z_i^{\frac{1}{m}},\z}(x_i \, \ve^{-1/2},\ve) \,.
$$
{\bf Claim 2:}
When $\lambda < -1/2$, then we have: 
\begin{align}
\label{eq.claim1}
\sum_{x \in (x^{(0)}+m \BZ)\sqrt{\ve}} e^{-\frac{1}{m} x^t \ti A x}
\varphi(x,\ve)
& \ap
\sum_{x \in (x^{(0)}+m \BZ)\sqrt{\ve}; \, |x_i| < \ve^{\lambda+\frac{1}{2}}} 
e^{-\frac{1}{m} x^t \ti A x}
\varphi(x,\ve)
\,.
\end{align}
{\bf Claim 3:}
When $\lambda > -2/3$ and $K \in \BN$, then we have:
\begin{align}
\label{eq.claim2}
e^{-\frac{1}{m} x^t \ti A x}
\varphi(x,\ve)
&\=
e^{-\frac{1}{m} x^t \ti A x}
\Bigl(1+\sum_{p=1}^K C_p(x)\ve^{p/2}\Bigr) + o(\ve^{K(3\lambda+2)} 
\end{align}
where $C_p(x)$ are polynomials defined by Proposition~\ref{prop.1}. 

\noindent
{\bf Claim 4:}
If $P$ is a polynomial and when $\lambda <-1/2$, then we have:
\begin{align}
\label{eq.claim3}
\sum_{x \in (x^{(0)}+m \BZ)\sqrt{\ve}; \, |x_i| < \ve^{\lambda+\frac{1}{2}}} 
P(x) e^{-\frac{1}{2m} x^t \ti A x} & \ap (m \ve)^{-N/2} \int_{\BR^N} P(x)
e^{-\frac{1}{m} x^t \ti A x} dx \,.
\end{align}
Note that there is a competition of the range of $\lambda$ in claims 
2 and 3, and it is fortunate that the allowable range is nonempty. 
All three claims~\eqref{eq.claim1}-\eqref{eq.claim3} follow from an
application of the Poisson summation formula explained in detail in 
p.623--625 of~\cite{VZ}. Let us elaborate a bit with some comments on 
Poisson summation focusing on Claim 4 which states, among other things, 
that the asymptotics of a sum over a shifted lattice is independent 
of the shift. 

\noindent
{\bf Poisson summation:}
Suppose that $\phi$ is a $C^\infty$-function with more than polynomial 
decay at infinity, i.e., $|\phi(x)| =o(|x|^K)$ for all $K>0$. 
Then, 
\be
\label{eq.ps}
\sum_{k \in \BZ} \phi((k+\a)\ve) \sim \frac{1}{\ve} 
\widehat\phi(0) \,.
\ee
The proof of~\eqref{eq.ps} follows from Poisson summation formula
$$
\sum_{k \in \BZ} \phi(k+\a) \= \sum_{\ell \in \BZ} \widehat\phi(\ell) \e(\ell \a)
$$
which implies that 
\begin{equation}
\label{eq.pse}
\sum_{k \in \BZ} \phi((k+\a)\ve) \= \frac{1}{\ve}
\sum_{\ell \in \BZ} \widehat\phi\Bigl(\frac{\ell}{\ve}\Bigr) \e(\ell \a) \,.
\end{equation}
Since $\phi$ is $C^\infty$, $\widehat\phi(x)=O(|x|^{-K})$ for every $K>0$
as $|x| \gg 0 $. Consequently, for $\ell \neq 0$, each term of the right hand 
side of~\eqref{eq.pse} is exponentially small and so is
the sum for all nonzero $\ell$. This proves~\eqref{eq.ps}.

To show Claim 4, we use the Poisson summation formula
$$
\sum_{x \in (x^{(0)}+m \BZ)\sqrt{\ve}}
P(x) e^{-\frac{1}{m} x^t \ti A x} = 
\sum_{x \in x^{(0)}+m \BZ}
P(x\sqrt{\ve}) e^{-\frac{\ve}{2m} x^t \ti A x} = \sum_{x \in m \BZ} g(x)
\e(x^t x^{(0)}) \sim g(0)
$$
where $g(x)$ denotes the Fourier transform of 
$P(x\sqrt{\ve}) e^{-\frac{\ve}{2m} x^t \ti A x}$. Since 
$$
g(0)=\ve^{-N/2} \int_{\BR^N} P(x) e^{-\frac{1}{m} x^t \ti A x} dx \,,
$$ 
and since the sum in Claim 4 for $|x_i|>\ve^{\lambda+\frac{1}{2}}$ is
$O(\ve^K)$ for all $K>0$, Claim 4 follows. In conclusion, we have shown that:
\be
\label{eq.fQzfinal}
f_{Q,\z}^{[k]}(\ve) \sim \frac{\chi^N}{m^{N/2}} \,c(Q)\,
\prod_{i=1}^N D_{\z}(\z\th_i)^{-\frac1m}\,
\prod_{i=1}^N\frac{\th_i^{(A k)_i}}{(\z\th_i;\z)_{k_i}}
\,I_{Q,\z}(k,\ve)
\ee
where $I_{Q,\z}(k,\ve)$ is given by~\eqref{eq.Ik}. The above proof
applies mutantis mutandis to the proof of Claim 1.

Combining Equations~\eqref{eq.fQzfinal} and~\eqref{eq.fQzcomb}
concludes the proof of Theorem~\ref{thm.1}.
\qed



\section{A syntactical identity among two collections of formal 
power series}
\label{sec.2series}


In this section we discuss a syntactical identity between two formal power
series at each complex root of unity, one introduced in~\cite{DG1} and
~\cite{DG2} to describe the conjectural asymptotics of the Kashaev invariant 
near $1$ and near general roots of unity, respectively, and the other
being the radial asymptotics of Nahm sums as $q$ tends to a root of 
unity~$\z$, as given in~\cite{Zagier} for $\z=1$ and in the present paper for
general~$\z$.  
We observed by chance that the asymptotic series found in~\cite{Zagier} and 
in~\cite{DG1} agreed to all orders. This then turned out to be true for 
all~$\z$, giving 
a surprising connection between radial asymptotics of $q$-series 
and asymptotics of quantum invariants defined at roots of unity 
that was highlighted in~\cite{Ga:arbeit} and further discussed 
in~\cite{GZ:qseries}. Formally, this connection can be expressed 
by the commutativity, for all roots of unity~$\z$, of the following diagram
\be
\label{eq.NZNdata}
\begin{diagram}
\node{\text{NZ data}} \arrow[2]{e,t}{T} \arrow{se,r}{(1)_\z} 
\node[2]{\text{Nahm data}} \arrow{sw,r}{(2)_\z} \\
\node[2]{\text{Power series}}
\end{diagram}
\ee
whose ingredients we now explain. 

Briefly, Neumann-Zagier (in short, NZ) data are obtained from an ideal
triangulation of a cusped hyperbolic 3-manifold $M$ triangulated with $N$
tetrahedra with shapes 
$z=(z_1,\dots,z_N) \in \BC\setminus\{0,1\}$~\cite{NZ,Th}. The shapes
satisfy the NZ equations, which have the form
\be
\label{eq.NZ}
z^{\mb A} z''^{\mb B}=e^{ - \pi i \eta}
\ee
where for a complex number $w$, we define $w'=1/(1-w)$, $w''=1-1/w$ and 
$(\mb A \,\, \mb B)$ is the upper half of a symplectic matrix (i.e.,
it has full rank and $\mb A \mb B^t$ is symmetric) and $\eta \in \BZ^N$. 
A solution of~\eqref{eq.NZ} in $\BC\setminus\{0,1,\infty\}$ 
gives rise to a $\PSL(2,\BC)$-representation of the fundamental group of $M$
and describes the complete hyperbolic structure of $M$ when 
the solution is in the upper half plane (i.e., $\Im(z_i)>0$ for all $i$).
The NZ equations are written for each edge of the triangulation, and for
a choice of (meridian-longitude) peripheral curves of each boundary 
component of $M$. When $M$ has a single torus boundary component, equipped
with a meridian and longitude, the matrices $\mb A$ and $\mb B$ discussed 
in~\cite{DG1} we obtained by eliminating the shape $z'$ (using the 
fact that $z z' z''=-1$) giving rise to the vector $\eta$ in~\eqref{eq.NZ},
and by removing one the edge equations and replacing it by a meridian gluing 
equation. In addition, a flattening $f \in \BZ^N$ was introduced and 
used in~\cite{DG1}.

The map $T$ that appears in~\eqref{eq.NZNdata} converts the NZ equation to
a Nahm equation. Assuming that $\mb B$ is nonsingular, we can formally
convert~\eqref{eq.NZ} in the following form

\begin{equation}
\label{eq.nahmt}
1 - z_i \=  \e(B.e_i) \prod_{j=1}^N z_j^{A_{ij}} \qquad (i=1,\dots,N) \,.
\end{equation}
where $e_i$ is the $i$th coordinate vector, $\e(x)$ is as in~\eqref{eq.ex}
and 
\begin{equation}
\label{eq.ABchi}
A=I-\mb B^{-1} \mb A, \qquad B=\frac{1}{2}(-\mb B^{-1}\eta+(1,\dots,1)^t) \,.
\end{equation}
Since $(\mb A \, \mb B)$ is the upper half of a symplectic matrix, it
follows that $A$ is symmetric. This motivates the map $T$ 
from NZ-data to Nahm-data
\begin{equation}
\label{eq.NZ2nahm}
T(\mb A, \mb B, \eta, f) \mapsto (A, B, C) 
\end{equation}
where $C=f$. The transformed equation~\eqref{eq.nahmt} is the Nahm
equation of a twisted Nahm sum $F^*$ defined by:

\begin{equation}
\label{eq.FABCD}
F^*_{A,B,C}(q) =
\sum_{n\in\Z_{\ge0}^N} \e(B.n) 
\frac{q^{\frac12 n^tAn + B.n+C}}{(q)_{n_1}\cdots(q)_{n_N}}\, \in 
\BC(\!(q^{\frac{1}{d}})\!)
\end{equation}
where $A=(a_{ij})$ is a symmetric positive definite $N \times N$ matrix with 
rational entries, $B\in \BQ^N$ are column vectors and $C \in \BQ$ a scalar.

Now
fix a primitive root of unity $\z$. The arrow $(1)_\z$ in~\eqref{eq.NZNdata}
is a power series defined in~\cite{DG2} (under the hypothesis that
$\calH=-\mb B^{-1}\mb A+\text{diagonal}(z')$ is invertible), and the arrow 
$(2)_\z$ 
is the formula of Theorem~\ref{thm.1} applied formally to the twisted 
Nahm sum $F^*_{A,B,C}(q)$ as $q\to\z$. The reason for the commutativity of the
diagram~\eqref{eq.NZNdata} is that in~\cite{DG1,DG2},
the formal power series~\eqref{eq.l1} appears due to asymptotic expansion
of Faddeev's quantum dilogarithm. The latter is a ratio of two infinite
quantum factorials, one in the variable $q=\e(\tau)$ and the other in
the variable $\tilde{q}=\e(-1/\tau)$. Ignoring one of the infinite 
quantum factorials produces identical power series after formal Gaussian
integration.


\section{Coefficient versus radial asymptotics of Nahm sums}
\label{sec.coeff}



In this section we discuss a 
relation between the coefficient and the radial asymptotics of an analytic 
function in the complex unit disk under some fairly weak analytic 
assumptions which (for instance) are satisfied for the Nahm 
sums~\eqref{eq.FABC}.

Consider a function
\be
\label{eq.an}
G(q) \= \sum_{n=0}^\infty c(n) q^n \,.
\ee
analytic function in the open complex unit disk $|q|<1$ and with  
an asymptotic expansion at $q=1$ 
\be
\label{eq.radG}
G(e^{-z}) \;\sim\; e^{C^2/(4z)} 
\sum_{\alpha}^\infty A_\alpha z^\alpha \,,
\ee
as $z \to 0$ with $\Re(z)>0$ where $C$ is a positive real number and 
$\alpha$ is a sequence of real numbers tending to infinity. In the case 
of a Nahm sum, $\alpha \in \BN$, and in most applications, $\alpha$ lies in 
a fixed number of arithmetic progressions of the form 
$\alpha_0+\frac{1}{d}\BN$ for $\alpha_0 \in \BQ$ and $d \in \BN$. 
Assume further that for every $N >0$, there exists $\th_N>0$ such that
$\th_N=o(N)$ and $|G(e^{-h + i\th})| < h^N e^{C^2/(4h)}$ for $h>0$ (and small)
and $|\th| > |\th_N|$. 

\begin{theorem}
\label{thm.2ass}
Under the above assumptions, we have: 
\be
\label{eq.circle}
c(n) \;\sim\; \frac{1}{2} \sqrt{\frac{C}{2\pi}} \, \frac{1}{n^{3/4}} \,
e^{C\sqrt{n}} 
\sum_{\ell \geq 0, \, \alpha} 
\frac{(-1)^\ell}{C^{\ell-\alpha}} 
(2\ell-1)!! \binom{\alpha+\ell-\frac{1}{2}}{2\ell} 
\frac{1}{n^{\frac{\ell}{2}+\frac{\alpha}{2}}} \, A_\alpha 
\ee
\end{theorem}
Note that this implies that the asymptotics of the Fourier coefficients 
$c(n)$ determine the radial asymptotics $G(e^{-h})$ and vice-versa.


\begin{proof}
The Cauchy residue theorem and the change of variables $q=e^{-z}$ for
$z=h + \pi i \th$ for $\th \in [-1,1]$ and $h>0$ fixed implies that
$$
c(n) \= \frac{1}{2 \pi i} \int_{h-\pi i}^{h+\pi i} e^{n z} G(e^{-z}) dz \,. 
$$
Using the change of variables $z=Cu/(2\sqrt{n})$ 
it follows that
$$
c(n) \;\sim\; \sum_\alpha A_\alpha \, K(\alpha,n)
$$
where the accuracy of the approximation will depend on the accuracy and 
uniformity of~\eqref{eq.radG} and
\begin{align}
\label{eq.Kbetan}
K(\alpha,n) &= \frac{1}{2 \pi i} \int_{h-\pi i}^{h+\pi i} 
\exp\left(\frac{C^2}{4z} + nz \right) z^\alpha dz \\ \notag 
&=\frac{1}{2 \pi i} \biggl(\frac{C}{2\sqrt{n}}\biggr)^{\alpha+1}
\int \exp\left(\frac{C\sqrt{n}}{2} \bigl(u+\frac{1}{u}\bigr)\right) 
u^\alpha du \,.
\end{align}
The function $K(\alpha,n)$ can be written in closed form in terms of the
modified Bessel function as follows $K(\alpha,n) = K_{\alpha+1}(-C\sqrt{n})$,
and the latter has a well-known asymptotic expansion but since the direct
calculation of the asymptotic expansion is not difficult, we give it 
completely here. Make the substitution 
\be
\label{eq.xz}
u+\frac{1}{u} \= 2-x^2, \qquad u\= \bigl(\sqrt{1-x^2/4} \+ ix/2 \bigr)^2
\ee
to make the exponential in~\eqref{eq.Kbetan} a pure Gaussian. Then using the
standard binomial coefficient identity
\be
\label{eq.binomind}
\frac{d}{dx} \Bigl( \frac{u^k}{k} \Bigr) =
\sum_{j=0}^\infty (-1)^j \binom{k+\frac{j+1}{2}}{j} x^j
\ee
with 
$k=\alpha+1$ together with the 
standard Gaussian integral~\eqref{eq.FGI2} for $j=2\ell$ even, we get that
\be
\label{eq.kbn}
K(\alpha,n) \;\sim\; 
\frac{C^{\alpha+\frac{1}{2}}}{2^{\alpha+\frac{3}{2}} \sqrt{\pi}}
e^{C \sqrt{n}} \sum_{\ell=0}^\infty \frac{(-1)^\ell}{C^\ell} 
(2\ell-1)!! \binom{\alpha+\ell-\frac{1}{2}}{2\ell} 
n^{-\frac{3}{4}-\frac{\ell}{2}-\frac{\alpha}{2}} \,.
\ee
This completes the proof.
\end{proof}



\section{Modular Nahm sums}
\label{sec.C}


In this section we give an applications of the asymptotic 
Theorem~\ref{thm.1} to the case when the Nahm sum $F_{A,B,C}(q)$ is 
modular. 

Let $X_A=(X_{A,1},\dots,X_{A,N}) \in (0,1)^N$ 
denote the distinguished solution of the Nahm equation $1-X=X^A$ and let
$\xi_A \in B(\BC)$ denote the corresponding element of the Bloch group and
set
\be
\label{eq.C0}
C_0(A)=-\L(\xi_A)/(2 \pi)^2
\ee
where $\L$ is our normalization of the Rogers dilogarithm given 
in~\eqref{eq.rogers}.

We first make some general remarks about modular functions and their 
asymptotic properties near rational points. First, by \emph{modular function} 
we will always mean a function invariant under a subgroup of finite index 
of~$\SL(2,\BZ)$. (We do not have to assume that this subgroup is a congruence
subgroup, i.e.~one containing the principle congruence subgroup~$\Gamma(M)$ 
for some~$M\in\N$, although in the case of Nahm sums, which always have an 
expansion in rational powers of~$q$ with integral coefficients, a well-known 
conjecture implies that if they are modular at all then they are in fact 
modular with respect to a congruence subgroup.)  For any such 
function~$g(\tau)$ and any $P \in \P^1(\BQ)=\BQ \cup \{\infty\}$, we define 
the valuation $v_P(g) \in \BQ$ of~$g$ at~$P$ as the smallest exponent of 
$q=\e(\tau)$ in the  Fourier expansion of $(g\circ\gamma)(\tau)$, where 
$\gamma \in \SL(2,\BZ)$ is any element such that $\gamma(\infty)=P$.  
This definition is easily seen to be independent of the choice of~$\gamma$.

Recall $f_Q(\tau)$ from~\eqref{eq.fQ}, and let $F_Q$ denote $F_{A,B,C}$.

\begin{proposition}
\label{prop.Cbound}
If $f_{Q}(\tau)$ is modular, then for every $P \in \P^1(\BQ)$ we have
\be \label{eq.VC}
v_P(F_{Q}) \;\geq\; C_0(A)
\ee
with equality when $P=0$.
\end{proposition}
\begin{proof} 
Let $f=f_{Q}$, $P=a/c$ for $(a,c)=1$, $c>0$ and 
$\gamma=\sm abcd\in \SL(2,\BZ)$. Take $\ep >0$ and set 
$\tau=(i\ep\i-d)/c$ in the upper half-plane ($\tau \to \infty$ as 
$\ep \to 0^+$). Then, 
$$
\gamma(\tau)=\frac{a\tau+b}{c\tau+d}=\frac{a}{c} + \frac{i\ep}{c}
$$
combined with $q=\e(\tau)=\e(i v_P(f)/(c \ep)) \e(-d v_P(f)/c)$
implies that
\be
\label{eq.fas1}
(f \circ \gamma)(\tau)=f\left(\frac{a+i\ep}{c}\right) 
= C' \,\, q^{v_P(f)}(1 + O(q^{1/D})) \;\sim\; 
C' \e(i v_P(f)/(c \ep))
\ee
for some $C' \neq 0$ and some $D \in \BQ_+$. On the other hand, 
Theorem~\ref{thm.1} with $\z=\e(P)$ and $n=c$ imply that
\be
\label{eq.fas2}
f\left(\frac{a+i\ep}{c}\right) \;\sim\; 
C'' \e(i C_0(A)/(c \ep))
\ee
where $C''$ is a constant, possibly zero. A comparison 
between~\eqref{eq.fas1} and~\eqref{eq.fas2} implies inequality~\eqref{eq.VC}.
When $P=0$, i.e., $\z=1$, Theorem~\ref{thm.1} asserts that
$C'' \neq 0$. In that case, \eqref{eq.fas1} and~\eqref{eq.fas2} imply equality
in~\eqref{eq.VC}.
\end{proof}
As a special case of the proposition, for $P=\infty$ it follows that
\be  \label{eq.CC0}
v_\infty(f_{Q}) \= \min_{n \in \BZ^N_{\geq 0}} 
\bigl( Q(n) \bigr) \;\geq\; C_0(A) 
\ee
whenever $f_{Q}$ is modular. If in addition $ \frac12n^tAn+n^tB\geq 0$ for all 
$n \in \BZ^N_{\geq 0}$ (as is the case for all the modular triples
$(A,B,C)$ of rank~1, 2 or~3 listed in~\cite{Zagier}
and~\cite{VZ}), then $v_\infty(f_{A,B,C})=C$ and we deduce that $C\ge C_0(A)$. 
Moreover, in all cases observed, the equality 
$C=C_0(A)$ holds if and only if the vector~$B$ is zero, and this value occurs 
whenever the matrix~$A$ is integral and even.  (The converse to this last 
statement, however, is not true; for instance, the Nahm sum $f_{A,0,C_0(A)}$ 
is modular also for $A=\sm{4/3}{2/3}{2/3}{4/3}$
or its inverse $A=\sm1{-1/2}{-1/2}1$, as well as for several non-integral 
$3\times3$ matrices~$A$.)


%


\appendix
\section{Application: proof of the Kashaev-Mangazeev-Stroganov
identity}
\label{sec.5term}

 
The current paper is needed crucially in~\cite{CGZ}, where the asymptotic
properties of Nahm sums at roots of unity are used to prove Nahm's 
conjecture about their modularity. A further essential ingredient 
in~\cite{CGZ} was the following finite version of the 5-term relation for
the cyclic quantum dilogarithm due to 
Kashaev, Mangazeev and Stroganov:

\begin{proposition}
\cite[Eqn.C.7]{KMS}
Let $X$, $Y$ and $Z$ be three complex numbers satisfying $Z=\frac{1-X}{1-Y}$
and $\z$ a primitive $m$th root of unity. Then 
\be
\label{eq.KMS}
\frac{D_\z(1) D_\z(y \z/x) D_\z(x/y z)}
{D_\z(1/x) D_\z(y \z) D_\z(\z/z)} \=
(\z y)^{m(1-m)/2} \, f(x,\,y\,|\,z)^m \,, 
\ee
where $x$, $y$ and $z$ are $m$th roots of $X$, $Y$ and $Z$ and 
\be
\label{eq.fxyz}
f(x,\,y\,|\,z) = \sum_{k \mod m} \frac{(\z y;\z)_k}{(\z x;\z)_k}\, 
z^k \,.
\ee
\end{proposition}
An independent proof of the above identity was given in unpublished work
of Gangl and Kontsevich. In this appendix we give a simple proof of this 
identity as an application of the asymptotic formula in Lemma~\ref{lem.1},
or rather of its weakening (take $w=x$, $\nu=0$ and retain only the leading
terms)
\begin{equation}
\label{tag1}
\bigl(x;\,\z e^{-\ve/m}\bigr)_\infty^{\;-m} \ap 
\frac{D_\z(x)}{(1-x^m)^{m/2}}\,e^{\Li_2(x^m)/\ve} 
\qquad (x \in \BC,\; x^m \not\in [1,\infty), \; \ve\searrow0) \,.
\end{equation}
in combination with a famous identity of Ramanujan. 
The same method could presumably be used to prove many other identities.

Note that the right hand side of~\eqref{eq.fxyz} is well-defined
because the relation $Z(1-Y)=1-X$ implies that the summand is $m$-periodic.
Furthermore, both sides of~\eqref{eq.fxyz} are rational functions on
the curve $z^m(1-y^m)=1-x^m$, 
so it suffices to prove them in an open set of that curve. With this 
in mind, let $X,\,Y,\,Z$ be as in the proposition above but also satisfying 
that $X, Y \not\in\BR$ and $|X/Y|<|Z|<1$. 
Set $x=X^{1/m}$, $y=Y^{1/m}$, and $z=Z^{1/m}$, and choose $q=\z e^{-\ve/m}$ 
with $\ve>0$ small.
The Ramanujan $\Ps$ summation formula says that
\begin{equation}
\label{tag2}
\Ps(x,y,z;q) \;:=\; \sum_{k=-\infty}^\infty \frac{(qy;q)_k}{(qx;q)_k}\,z^k
  \= \frac{(q;q)_\infty\,(qyz;q)_\infty\,(1/yz;q)_\infty\,(x/y;q)_\infty}
  {(qx;q)_\infty\,(1/y;q)_\infty\,(z;q)_\infty\,(x/yz;q)_\infty}\,,  
\end{equation}
where $(x;q)_k=(q^kx;q)_{|k|}^{\;-1}$ for~$k<0$ and where the series
converges because of the conditions placed on $X$, $Y$ and~$Z$. 
Denote by $A_k$ the $k$th summand in the series. Then for fixed~$k$ we have 
$$
\frac{A_{k+m}}{A_k} \= \prod_{j=1}^{m}\biggl(\frac{1-q^{k+j}y}{1-q^{k+j}x}z\biggr)
 \= \frac{1-Y}{1-X}\,Z \+ \O(\ve)\= 1 \+ \O(\ve) \qquad(\ve\searrow0)\,,
$$
so $A_k$ is periodic up to finite order in $\ve$. This implies that the
left-hand side of~\eqref{tag2} is the sum of $m$ terms each of the form
$\sum_{n \in \BZ} \phi(n)$ where $\phi(x)$ is an approximate Gaussian centered
at $x=0$. If we assume only $k=\text o(1/\ve)$ rather than $k=O(1)$, then
we have instead:
$$
\frac{A_{k+m}}{A_k} \= \frac{1-e^{-k\ve}Y}{1-e^{-k\ve}X}\,Z \+ \O(\ve)
  \= 1 + \Bigl(\frac Y{1-Y}\,-\,\frac X{1-X}\Bigr)k\ve \+ \O(\ve+k^2\ve^2)\,. 
$$
It follows that 
$$ 
A_{k+nm} \= \frac{(\z y;\z)_k}{(\z x;\z)_k}\,z^k\,\cdot\,
\exp\Bigl(\frac{X-Y}{(1-X)(1-Y)}\,\frac{n^2m}2\,\ve 
\+ \O(n\ve+n^3\ve^2)\Bigr)
$$
for $k$ fixed and $n=\text o(1/\ve)$. Our assumptions on $(X,Y,Z)$
imply that $\Re\bigl(\frac{Y-X}{(1-X)(1-Y)}\bigr)<0$, so
the right-hand side behaves like a Gaussian. We deduce that
\begin{equation}
\label{tag3} 
\Ps(x,y,z;\z e^{-\ve/m}) \ap \sqrt{\frac{2\pi}{\ve}}
\sqrt{\frac{(1-Y)(1-X)}{(X-Y)}} \,
f(x,\,y\,|\,z)  
\qquad(\ve\searrow 0), 
\end{equation}
where 
$f(x,\,y\,|\,z)$ as in Equation~\eqref{eq.fxyz}.
On the other hand, by the transformation formula of the Dedekind eta-function
the factor $(q;q)_\infty$ in~\eqref{tag2} satisfies the asymptotic formula
$$ 
(q;q)_\infty \=  (\mu+\O(\ve)) \sqrt{\frac{2\pi}{\ve}} \,e^{-\pi^2/6m\ve} 
$$
for some $(24m)$th root of unity~$\mu$, and inserting this and the
asymptotic formula from~\eqref{tag1} into the product in equation~\eqref{tag2} 
we find the alternative asymptotic formula
\begin{equation}
\label{tag4}
\Ps(x,y,z;\z e^{-\ve/m})^m \ap \sqrt{\frac{2\pi}{\ve}} \,C^{m/2}D \,e^{B/\ve} 
\qquad (\ve\searrow 0), 
\end{equation}
where 
$$ 
B\,=\,
-\Li_2(X/YZ)+\Li_2(YZ)+\Li_2(1/YZ)+\Li_2(X/Y)+\Li_2(1)
-\Li_2(X)-\Li_2(1/Y)-\Li_2(Z),
$$
$$ 
C \= \frac{(1-YZ)\,(1-1/YZ)\,(1-X/Y)}{(1-X)\,(1-1/Y)\,(1-Z)\,(1-X/YZ)}\,,
$$
and 
$$ 
D \= \mu_{24}\,\frac{D_\z(\z x)\,D_\z(1/y)\,D_\z(z)\,D_\z(x/yz)}{D_\z(\z yz)
\,D_\z(1/yz)\,D_\z(x/y)}\,,
$$
where $\mu_{24}=\z^m$ is a 24th root of unity. The quantity $B$ vanishes by 
the standard functional equations of the dilogarithm. 
Using further the identities
$$
D_\z(\z x) \= D_\z(x) \frac{(1-x)^m}{1-x^m}\,, \qquad  
D_\z(1/x) D_\z(x) \= \mu_6 x^{-\frac{m(m-1)}{2}} \frac{(1-x^m)^m}{(1-x)^m} 
$$
(where $\mu_6$ is a sixth root of unity) and comparing the asymptotic equations 
\eqref{tag3} and \eqref{tag4}, we get~\eqref{eq.KMS} as desired. 


\bibliographystyle{plain}
\bibliography{biblio}
\end{document}